\documentclass[12pt]{article}
\input epsf.tex

\usepackage{graphicx}
\usepackage{amsmath,amsthm,amsfonts,amscd,amssymb,comment,eucal,latexsym,mathrsfs}
\usepackage{stmaryrd}
\usepackage[all]{xy}

\usepackage{epsfig}

\usepackage[all]{xy}
\xyoption{poly}
\usepackage{fancyhdr}
\usepackage{wrapfig}
\usepackage{epsfig}

\theoremstyle{plain}
\newtheorem{thm}{Theorem}[section]
\newtheorem{prop}[thm]{Proposition}
\newtheorem{lem}[thm]{Lemma}
\newtheorem{cor}[thm]{Corollary}

\theoremstyle{definition}
\newtheorem{defn}{Definition}
\theoremstyle{remark}
\newtheorem{remark}{Remark}

\newtheorem{question}{Question}

\advance \topmargin by -\headheight
\advance \topmargin by -\headsep
\evensidemargin \oddsidemargin
    \def\E{{\mathbb{E}}} \def\F{{\mathbb{F}}}        \def\N{{\mathbb{N}}}  \def\P{{\mathbb{P}}}           \def\Z{{\mathbb{Z}}}

        \def\bfI{{\bf{I}}}   \def\bfL{{\bf{L}}}  \def\bfN{{\bf{N}}}    \def\bfR{{\bf{R}}} \def\bfS{{\bf{S}}}     \def\bfX{{\bf{X}}} \def\bfY{{\bf{Y}}} 

\def\bfa{{\bf{a}}} \def\bfb{{\bf{b}}}                     \def\bfw{{\bf{w}}}   

\def\cA{{\mathcal{A}}} \def\cB{{\mathcal{B}}}   \def\cE{{\mathcal{E}}} \def\cF{{\mathcal{F}}} \def\cG{{\mathcal{G}}}         \def\cP{{\mathcal{P}}}  \def\cR{{\mathcal{R}}} \def\cS{{\mathcal{S}}} \def\cT{{\mathcal{T}}}  \def\cV{{\mathcal{V}}} \def\cW{{\mathcal{W}}}   

\def\sA{{\mathscr{A}}} \def\sB{{\mathscr{B}}}                        

       \def\tH{{\tilde{\rm{H}}}}                  

\newcommand{\G}{\Gamma}
\newcommand{\Ga}{\Gamma}

\newcommand{\eps}{\epsilon}

\renewcommand\a{\alpha}
\renewcommand\b{\beta}
\renewcommand\d{\delta}

\renewcommand\k{\kappa}
\renewcommand\l{\lambda}
\renewcommand\o{\omega}
\newcommand\s{\sigma}

\def\bftom{{\mbox{\boldmath $\tilde{\omega}$}}}
\def\bfom{{\mbox{\boldmath $\omega$}}}
 
\def\bfcF{{\mbox{\boldmath $\cF$}}}
\def\bfphi{\mbox{\boldmath $\phi$}} 
\def\bfcT{{\mbox{\boldmath $\cT$}}}
\def\bfpsi{\mbox{\boldmath $\psi$}} 
 
\def\bftau{\mbox{\boldmath $\tau$}} 

\newcommand\Aut{\operatorname{Aut}}
\newcommand\FRI{{\operatorname{FRI}}}

\newcommand\Geom{{\operatorname{Geom}}}

\newcommand\Poisson{\operatorname{Poisson}}
\newcommand\Prob{\operatorname{Prob}}

\newcommand\Var{\operatorname{Var}}

\def\cc{{\curvearrowright}}
\newcommand{\resto}{\upharpoonright}

\begin{document}
\title{Finitary random interlacements and the Gaboriau-Lyons problem}
\author{Lewis Bowen\footnote{supported in part by NSF grant DMS-1500389} \\ University of Texas at Austin}
\maketitle

\begin{abstract}
The von Neumann-Day problem asks whether every non-amenable group contains a non-abelian free group. It was answered in the negative by Ol'shanskii in the 1980s. The measurable version (formulated by Gaboriau-Lyons) asks whether every non-amenable measured equivalence relation contains a non-amenable treeable subequivalence relation. This paper obtains a positive answer in the case of arbitrary  Bernoulli shifts over a non-amenable group, extending work of Gaboriau-Lyons. The proof uses an approximation to the random interlacement process by random multisets of geometrically-killed  random walk paths. There are two applications: (1) the Gaboriau-Lyons problem for actions with positive Rokhlin entropy admits a positive solution, (2) for any non-amenable group, all Bernoulli shifts factor onto each other.
\end{abstract}


\noindent
\tableofcontents

\section{Introduction} 


Amenability is inherited by subgroups and nonabelian free groups are non-amenable. These observations lead to the von Neumann-Day problem: does every non-amenable group contain a  nonabelian free group? It was disproven by Ol'shankii \cite{olshankii-book}. However, the analogous problem for measured equivalence relations remains open. This article focuses on the special case of orbit-equivalence relations arising from free actions of countable groups:
\begin{question}[Gaboriau-Lyons \cite{gaboriau-lyons}]
Suppose $\G$ is a countable group, $(X,\cB,\mu)$ a standard probability space and $\G \cc (X,\cB,\mu)$ is an essentially free, ergodic, measure-preserving action. If $\G$ is non-amenable then does there necessarily exist an essentially free, ergodic action $\F_2 \cc (X,\cB,\mu)$ of the rank 2 free group such that each orbit of the $\F_2$ action is contained in an orbit of the $\G$ action (after neglecting a measure zero set)?
\end{question}

An action $\G \cc (X,\cB,\mu)$ is {\bf von Neumann-Day} (vND) if either $\G$ is amenable or the above question admits a positive answer. An action is {\bf weakly von Neumann-Day} (weakly vND) if it satisfies the same condition as vND with the exception that the $\F_2$ action is allowed to be non-ergodic. So the Gaboriau-Lyons problem asks whether every essentially free, ergodic action of a countable group is vND\footnote{The problem whether every ergodic measured equivalence relation is weakly vND was formulated by D. Gaboriau as Question 5.16 of his HDR thesis \cite{gaboriau-thesis}.}. 

Important progress was made by Gaboriau-Lyons \cite{gaboriau-lyons}. Their result concerns Bernoulli shifts which are defined as follows. Let $K$ be a standard Borel space and $\k$ a Borel probability measure on $K$. Consider the product space $K^\G=\{x:\G \to K\}$ with the product sigma-algebra and the product measure $\k^\G$. The group $\G$ acts on $K^\G$ by shifting: 
  $$(gx)(f)=x(g^{-1}f), \quad g, f\in \G, x\in K^\G.$$
  The action $\G \cc (K,\k)^\G$ is the {\bf Bernoulli shift over $\G$ with base space $(K,\k)$}. 
  
    If $\k$ is supported on a countable subset of $K$ then the {\bf Shannon entropy} of $\k$ is defined by
 $$H(\kappa):=-\sum_{k\in K} \kappa(k) \log\kappa(k).$$
 If $\kappa$ is not supported on a countable set then $H(\kappa)$ is defined to be $+\infty$.  In \cite{bowen-jams-2010} and \cite{MR2813530} it is shown that if $\G$ is sofic then the Shannon entropy of the base space is a measure-conjugacy invariant. In \cite{seward-ornstein} Seward proved that, for any countable group $\G$, if two base spaces have the same Shannon entropy then the corresponding Bernoulli shifts are isomorphic. Special cases of this result were previously obtained by Ornstein \cite{ornstein-1970a} (for $\G=\Z$), Stepin \cite{stepin-1975} (for groups containing an infinite cyclic subgroup), Ornstein-Weiss \cite{OW80} (for amenable groups), and the author  \cite{bowen-ornstein-2012} (for all countable groups and base spaces of cardinality $>2$). All of the above works build on the fundamental work of Ornstein \cite{ornstein-1970a}.

  Gaboriau and Lyons proved that for every non-amenable group $\G$ there exists some probability space $(K,\kappa)$ such that the Bernoulli shift $\Ga \cc (K,\kappa)^\Ga$ is von Neumann-Day \cite{gaboriau-lyons} (see also the expository article \cite{MR3051202} for further applications). The main result of this paper is that the Gaboriau-Lyons Theorem holds for {\em all} nontrivial $(K,\kappa)$.
 
\begin{thm}\label{thm:main2}
Let $\Ga$ be a countable non-amenable group and $(K,\kappa)$ a nontrivial probability space. Then the Bernoulli shift  $\G \cc (K,\kappa)^\G$ is von Neumann-Day.
\end{thm}

\subsection{Applications}

\subsubsection{Positive entropy actions}
In spectacular recent work, Seward generalized Sinai's Factor Theorem to all countable groups \cite{seward-sinai-30}. Combined with  Theorem \ref{thm:main2}, this implies:

\begin{cor}\label{thm:main3}
Let $\Ga$ be a countable non-amenable group and $\Ga \cc (X,\mu)$ a probability-measure-preserving ergodic essentially free action with positive Rokhlin entropy. Then this action is weakly von Neumann-Day. 
\end{cor}
The definition of Rokhlin entropy and further details are explained in \S \ref{sec:app1}. I conjecture that the actions in Corollary \ref{thm:main3} are actually vND. 


\subsubsection{Factors of Bernoulli shifts}

An action $\G \cc (X,\mu)$ {\bf factors onto} another action $\G \cc (Y,\nu)$ if there is a $\G$-equivariant measurable map $\Phi:X \to Y$ with $\nu = \mu \circ \Phi^{-1}$. Two actions are  {\bf weakly isomorphic} if they factor onto each other. 

If $\G$ is amenable then the entropy of a factor action is bounded above by the entropy of the source. This is false for non-amenable groups.   In fact Ornstein and Weiss exhibited by explicit example that the 2-shift over the rank 2 free group $\F_2$ factors onto the 4-shift \cite{OW87}. In \cite{ball-factors1}, K. Ball generalized this example to show that if $\G$ is any non-amenable group then there is some Bernoulli shift $\G \cc (K,\k)$ with $K$ finite that factors onto all Bernoulli shifts over $\G$. In \cite{bowen-ornstein-2011} I proved that if $\G$ contains a non-abelian free group then all Bernoulli shifts over $\G$ are weakly isomorphic. The techniques of these papers combined with \cite{seward-ornstein} and Theorem \ref{thm:main2} prove:

\begin{cor}\label{thm:main4}
Let $\G$ be any countable non-amenable group. Then all Bernoulli shifts over $\G$ are weakly isomorphic. 
\end{cor}


Seward's Factor Theorem  \cite{seward-sinai-30} and Corollary \ref{thm:main4} imply:

\begin{cor}\label{thm:main5}
Let $\Ga$ be a countable non-amenable group and $\Ga \cc (X,\mu)$ a probability-measure-preserving, ergodic, essentially free action with positive Rokhlin entropy. Then this action factors onto all Bernoulli shifts over $\G$. 
\end{cor}


\subsection{A reduction}

Gaboriau-Lyons  \cite{gaboriau-lyons} uses the theory of cost to reduce their theorem to showing that a certain random subgraph of a Cayley graph of $\G$ (whose law is measurably conjugate to a Bernoulli shift) is such that it has infinitely many infinite connected components with infinitely many ends a.s. Moreover it has indistinguishable infinite clusters in the sense of Lyons-Schramm \cite{lyons-schramm-indistinguishability}. We will use the same reduction. However, since the random graph we use is obtained from a nontrivial factor of a Bernoulli shift, this only shows the weak von Neumann-Day property. To finish the proof, we use the Chifan-Ioana Theorem that Bernoulli shifts over non-amenable groups are solidly ergodic \cite{MR2647134}. 

\subsection{Finitary Random Interlacements} 

\subsubsection{The big picture}
Theorem \ref{thm:main2} is obtained by studying random multisets of finite random walks on a fixed Cayley graph of $\G$. These random multisets are called Finitary Random Interlacements (FRI). The union of the random walk paths in an FRI is a random subgraph that satisfies the aforementioned reduction when certain parameter bounds hold. Moreover, the action of $\G$ on the measure-space defining the $\FRI$ is measurably conjugate to a Bernoulli shift. 

The FRI is a variant of the random interlacement process (RI) which is a random multiset of bi-infinite random walk paths  introduced by Sznitman \cite{MR2680403}. The paper \cite{MR3076674} proves that, for non-amenable graphs, the RI at low intensity has infinitely many infinite components. It is unknown whether the RI at low-intensity is a factor of a low-entropy Bernoulli shift. If it were true, it could be used to give another proof of Theorem \ref{thm:main2}. It was this observation that led to the approach of this paper. Moreover, we use essentially the same proof as in \cite{MR3076674} to show that a low-intensity FRI has infinitely many infinite components whenever it has infinite components.

\subsubsection{An informal description}
The FRI is defined on a locally finite graph $G=(V,E)$. There are two parameters, $u,T>0$ called the {\bf intensity} and {\bf average stopping time}. To each vertex $x\in V$ is associated a Poisson random variable $\bfN_x$ with mean $u \deg_x/(T+1)$ where $\deg_x$ is the degree of $x$. Informally, we think of $\bfN_x$ as the number of frogs that live on vertex $x$ at time $0$. Each frog has a coin that lands on heads with probability $p=\frac{T}{T+1}$. At time 0, a frog flips her coin. If it lands on heads, the frog moves to a uniformly random neighboring vertex. It repeats this operation until the coin lands on tails at which point the frog stops forever. So each frog performs a simple random walk for $t$ steps where $t+1$ is a geometric random variable with mean $T+1$. The FRI is the random multiset of these finite random walk paths. 

It might help the reader's intuition to know that the expected number of walks of the $\FRI$ that traverse a fixed vertex $x$ is $u\deg_x P^{(T)}_x(\tH_x=\infty)$ where $ P^{(T)}_x(\tH_x=\infty)$ is the probability that the aforementioned geometrically-killed random walk started at $x$ does not return to $x$. In particular, if $G$ is vertex-transitive then this expected number of walks is proportional to $u$. 

To be more precise, let $\cW[0,\infty)$ denote the collection of all finite walks $w:D_w \to V$ where $D_w =[0, \textrm{len}(w)] \cap \N$ for some $\textrm{len}(w) \in \N$ called the length of $w$. By convention $\N=\{0,1,2,\ldots\}$. Then $\cW[0,\infty)$ is a countable set because $V$ is countable. Let $\Omega^{[0,\infty)}$ be the set of all locally finite $\N\cup\{\infty\}$-valued measures on $\cW[0,\infty)$. We can identify $\Omega^{[0,\infty)}$ with a subspace of the space of all functions from $\cW[0,\infty)$ to $\N$ endowed with the topology of pointwise convergence. Finally the law of the $\FRI$ is a probability measure $\P_{u,T}$ on $\Omega^{[0,\infty)}$. 

Given $\omega \in \Omega^{[0,\infty)}$, let $\cE_\o\subset E$ denote the collection of edges crossed by at least one walk in the support of $\o$. Recall that a {\bf cluster} is a connected component. Our main result is:

\begin{thm}\label{thm:main1}
Let $\G$ be a non-amenable countable group with a finite generating set $S \subset\G$.  Let $G=(V,E)$ denote the associated Cayley graph. Then:
\begin{enumerate}

\item for every $\eps>0$ there exists $u_0>0$ such that if $0<u<u_0$ and $T>0$ then the action $\G \cc (\Omega^{[0,\infty)}, \P_{u,T})$ is measurably conjugate to a Bernoulli shift over $\G$ with base entropy $<\eps$;

\item for every $u>0$ there exists $T_u$ such that if $T>T_u$ then for $\P_{u,T}$-almost every $\omega$, the subgraph $(V,\cE_\omega)$  has infinite clusters;

\item there exists $u_c>0$ such that if $0<u<u_c$ and $T>T_u$ then for  $\P_{u,T}$-almost every $\omega$, the subgraph $(V,\cE_\omega)$ has infinitely many infinite clusters;

\item for any $u,T>0$, if $(V,\cE_\omega)$ has infinitely many infinite clusters $\P_{u,T}$-a.s. then each infinite cluster has infinitely many ends $\P_{u,T}$-a.s. Moreover if $\bfom \sim \P_{u,T}$ then the random subgraph $(V,\cE_\bfom)$  has indistinguishable infinite clusters in the sense of Lyons-Schramm \cite{lyons-schramm-indistinguishability}. 

\end{enumerate}
\end{thm}

\begin{question}
Does item (2) hold when $\G$ is amenable and simple random walk on its Cayley graph is transient? In particular, does it hold on $\Z^d$?
\end{question}

\begin{remark}
Parts (2) and (3) of the theorem can be extended to all graphs $G$ that are bounded degree, connected and non-amenable. The conclusion of item (4) holds a.s. if $G$ is a unimodular network. For clarity's sake, the theorem is proven first for Cayley graphs, and then it is explained how minor changes give this generalization.
\end{remark}

\subsubsection{A brief sketch}

Theorem \ref{thm:main1} is proven in several steps. The first statement  follows from direct entropy computation. To prove the existence of infinite clusters, it is shown that the cluster containing the origin stochastically dominates a certain growth process which is analyzed with a second moment argument. To prove that infinite clusters and low-intensity implies the existence of  infinitely many infinite components, it is shown that the subgraph of $(V,\cE_\o)$ containing the origin is stochastically dominated by a branching random walk with $1+O(u)$ mean offspring distribution. This is similar to the way Benjamini-Schramm obtained bounds on  the non-uniqueness phase in Bernoulli percolation \cite{MR2883387} and the way Teixeira-Tykesson proved that low-intensity Random Interlacements have infinitely many components in non-amenable graphs \cite{MR3076674}. 

Lastly, it is observed if $\bfom \sim \P_{u,T}$ then the random subgraph $(V,\cE_\bfom)$ is insertion-tolerant. So by \cite{lyons-schramm-indistinguishability}, whenever $G$ is vertex-transitive and $(V,\cE_\bfom)$ has infinitely many infinite clusters, the infinite clusters are indistinguishable and have infinitely many ends a.s.


\subsection{Connections with Random Interlacements}

Random Interlacements (RI) are random multisets of bi-infinite random walk paths on a locally finite graph $G=(V,E)$. They were introduced by Sznitman \cite{MR2680403} where the focus is on the connectedness properties of the complement in the case when $G$ is a cubic lattice. See also \cite{MR3308116} for a general introduction and \cite{MR2525105} for a definition of RI on general weighted transient graphs. 

The `local picture' of Finitary Random Interlacements is similar to standard constructions of Random Interlacements, which I learned from \cite{MR3076674}. In the appendix it is shown that the FRI process  converges in distribution to the RI as $T \to\infty$. This result is not needed for any of the other results; we have included it only to justify the name FRI.  



\subsection{Related Literature}

There is an excellent Bourbaki article on the Gaboriau-Lyons Theorem and its applications to measurable group theory \cite{MR3051202}.

As far as I am aware, before the Gaboriau-Lyons Theorem there was only one technique for proving the existence of a free subgroup (measurable or otherwise): the ping-pong lemma used in the proof of the Tit's Alternative \cite{tits}. That idea was used in \cite{bowen-hyperbolic} to prove that measured equivalence relations that act properly on a bundle of hyperbolic spaces are von Neumann-Day.

The Gaboriau-Lyons Theorem was extended in \cite{MR3813200} to class-bijective extensions of measured equivalence relations. The techniques of that paper can be combined with this article to strengthen the main result of \cite{MR3813200} so that it holds for all Bernoulli shifts. 

In \cite{kun-2013}, Gabor Kun obtains a Lipschitz version of the Gaboriau-Lyons Theorem. I do not know if there exists a Lipschitz version of Theorem \ref{thm:main2}. 

{\bf Acknowledgements}. I am deeply grateful to Itai Benjamini for a series of email conversations on random interlacements which initiated this approach. Also thanks to the authors of \cite{MR3076674} from which I learned most of what I know about random interlacements. Thanks also to the anonymous referee for suggestions that greatly simplified several proofs.

\section{Notation} 

Throughout this paper, $\N=\{0,1,2,\ldots\}$. The notation $A \Subset B$ means that $A$ is a finite subset of $B$. 

We use boldface to denote random variables. For example, suppose $\P$ is a probability measure on a space $\Omega$. Then $\omega \in \Omega$ denotes an element of $\Omega$ while $\bfom \sim \P$ denotes a random variable taking values in $\Omega$ with law $\P$.   In other words, for any measurable event $E \subset \Omega$, $\P(\bfom \in E)=\P(E)$. Also $\bfN \sim \Poisson(m)$ indicates that $\bfN$ is a Poisson random variable with mean $m$ and $\bfN+1 \sim \Geom(T)$ indicates that $\bfN+1$ is a geometric random variable with expected value $T$. Finally, $\bfN_1 \sim \bfN_2$ means that $\bfN_1$ and $\bfN_2$ are identically distributed random variables. If $\omega \in \Omega$ then we may write $\P(\omega)$ instead of $\P(\{\omega\})$ for simplicity.



For any measure $\k$ on a set $X$ and any measurable map $\Phi:X \to Y$ the {\bf pushforward measure} $\Phi_*\k$ on $Y$ is defined by $\Phi_*\k = \k \circ \Phi^{-1}$. 

Let $(X,\cB,\mu)$ be a Borel space with a sigma-finite measure $\mu$. The {\bf Poisson point process} with intensity measure $\mu$ is a random $\N \cup\{\infty\}$-valued measure $\bfom$ on $X$ satisfying
\begin{itemize}
\item if $A_1,A_2,\ldots \subset X$ are disjoint measurable sets then the restrictions $\{\bfom \resto A_i\}_i$ are jointly independent;
\item for any measurable $A \subset X$ with $\mu(A)<\infty$, $\bfom(A)$ is a Poisson random variable with mean $\mu(A)$.
\end{itemize}
Thus if $\Omega$ is the set of all $\N \cup \{\infty\}$-valued measures on $X$ endowed with the sigma-algebra generated by the functions $\omega \mapsto \omega(A)$ (for $A \subset X$ measurable) then the law of  the Poisson point process $\bfom$ is the unique probability measure $\P$ on $\Omega$ satisfying
\begin{enumerate}
\item if $A_1,A_2,\ldots \subset X$ are disjoint measurable sets and $\cA_i$ is the sigma-algebra on $\Omega$ defined by restricting to $A_i$ then the sigma-algebras $\cA_1,\cA_2,\ldots$ are jointly independent with respect to $\P$,
\item for any measurable $A\subset X$ with $\mu(A)<\infty$ and $n \in \N$,
$$\P(\{\omega \in \Omega:~\omega(A)=n\}) = \exp(-\mu(A)) \frac{ \mu(A)^n}{n!}.$$
\end{enumerate}

\section{Finitary random interlacements}\label{sec:second}

We will use notation as in \cite[\S 3.2]{MR3773383}. Let $G=(V,E)$ be a locally finite connected graph on countably many vertices. For each $-\infty\le m \le n \le \infty$, let $L(m,n)$ be the graph with vertex set $\{i \in \Z:~ m \le i \le n\}$ and edge set $\{(i, i + 1) : m \le i \le n-1\}$. Let $\cW[m, n]$ be the set of multigraph homomorphisms $w:L(m, n) \to G$ such that the pre-image of each vertex in $G$ is finite, and let $\cW$, $\cW[0,\infty)$ be the unions
\begin{eqnarray*}
\cW&=&\bigcup_{-\infty\le m \le n \le \infty} \cW[m,n]\\
\cW[0,\infty)&=&\bigcup_{0\le n < \infty} \cW[0,n].
\end{eqnarray*}
For $w \in \cW[m,n]$, let $\textrm{len}(w)=n-m$ denote its length and $D_w = \{m,\ldots, n\}$ its domain.


 \begin{defn}
 For $x \in V$ and $0\le n \le \infty$, let $P^n_x$ be the law of the simple random walk started from $x$ and stopped at time $n$. More precisely, $P^n_x$ is the Borel probability measure on $\cW$ satisfying
 \begin{displaymath}
 P^n_x(w) = \left\{ \begin{array}{cl}
 \prod_{k=0}^{n-1} \deg_{w(k)}^{-1} & \textrm{if } w\in \cW[0,n] \textrm{ and } w(0)=x\\
 0 & \textrm{ otherwise} \end{array} \right.\end{displaymath}
For simplicity, we write $P_x$ for $P^\infty_x$.   Let
$$P^{(T)}_x = \left(\frac{1}{T+1} \right)\sum_{n \in \N} \left(\frac{T}{T+1} \right)^n P^n_x$$
denote the geometrization of the measures $P^n_x$. Thus $P^{(T)}_x$ is the law of the simple random walk started from $x$ with random stopping time $\sim \Geom(T+1)-1$. 
 \end{defn}
 
 A {\bf counting measure} is a measure taking values in $\N \cup \{\infty\}$. Given  $K \subset V$, let $\cW_K$ be the set of all walks $w\in \cW$ such that $w(i) \in K$ for some $i$.  A measure $\omega$ on $\cW$ is {\bf locally finite} if $\omega(W_K) <\infty$ for every finite $K \subset V$. 
 
 Since $\cW[0,\infty)$ is countable, we give it the $\s$-algebra of all subsets.  Let $\Omega^{[0,\infty)}$ denote the space of all locally finite counting measures $\omega$ on $\cW[0,\infty)$. We can think of $\Omega^{[0,\infty)}$ as a subset of the set of all functions from $\cW[0,\infty)$ to $\N$ (which is denoted by $\N^{\cW[0,\infty)}$). We give  $\N^{\cW[0,\infty)}$ the topology of pointwise convergence under which it is a Polish space. Because $\Omega^{[0,\infty)}$ is a $G_\delta$ subset of $\N^{\cW[0,\infty)}$, $\Omega^{[0,\infty)}$ is also Polish.

\begin{defn}[Finitary random interlacements]
For $0< T < \infty$, let $\nu^{(T)}$ be the measure on $\cW$ defined by
$$\nu^{(T)} = \sum_{x \in V} \frac{\deg_x}{T+1} P^{(T)}_x.$$
This measure is infinite if $V$ is infinite. For $0<u,T<\infty$, the {\bf finitary random interlacement}, typically denoted by $\bfom$, is the Poisson point process on $\cW$ with intensity measure $u \nu^{(T)}$. The law of $\bfom$ is a probability measure on $\Omega^{[0,\infty)}$ denoted by $\P_{u,T}$. To be precise, this means that for any $w_1,\ldots, w_k \in \cW[0,\infty)$ and natural numbers $n_1,\ldots, n_k$,
\begin{eqnarray*}
\P_{u,T}( \{ \omega \in \Omega^{[0,\infty)}:~ \omega(w_i) = n_i ~\forall 1\le i \le k\}) &=& \P_{u,T}( \bfom(w_i) = n_i ~\forall 1\le i \le k) \\
&=& \prod_{i=1}^k \exp(-u\nu^{(T)}(w_i)) \frac{(u\nu^{(T)}(w_i))^{n_i}}{n_i!}.
\end{eqnarray*}
In other words, $\bfom$ is a random element of $\Omega^{[0,\infty)}$ whose law is uniquely determined by:
\begin{enumerate}
\item the family of random variables $\{\bfom(w) : w\in \cW[0,\infty)\}$ is jointly independent;
\item for each $w\in \cW[0,\infty)$, $\bfom(w)$ is a Poisson random variable with mean $u \nu^{(T)}(w)$. 
\end{enumerate}

\end{defn}


\subsection{Bernoulli shifts}

This section proves part (1) of Theorem \ref{thm:main1}. So let $\G$ be a countable group with a finite symmetric generating set $S$ and let $G=(V,E)$ be the associated Cayley graph. To be precise, $V=\G$ and $E$ is the set of all unordered pairs $\{g,gs\}$ for $g\in \G, s\in S$. Because $\G$ acts on $V$ by left-translation, $\G$ also acts on $\cW[0,\infty)$ by
$$(gw)(n) = gw(n)\quad \forall g\in \G, w\in \cW[0,\infty), n \in D_w$$
and $\G$ acts on $\Omega^{[0,\infty)}$ by $(g\omega)(E) = \omega(g^{-1}E)$. The group $\G$ is not required to be non-amenable for the next result.

\begin{thm}\label{thm:main1.1}
With notation as above, for every $\eps>0$ there exists $u_\eps>0$ such that if $u<u_\eps$ then $\G \cc (\Omega^{[0,\infty)}, \P_{u,T})$ is measurably conjugate to a Bernoulli shift with base entropy at most $\eps$. 
\end{thm}

\begin{proof}
Let $\cW_g[0,\infty) \subset \cW[0,\infty)$ be the set of all finite walks on $G$ that start at $g \in V=\G$. Let $\Omega_g^{[0,\infty)}$ be the set of all $\omega \in \Omega^{[0,\infty)}$ such that $\omega$ is finite and supported on $\cW_g[0,\infty)$. Both $\cW_g[0,\infty)$ and $\Omega_g^{[0,\infty)}$ are countable sets. 

 Let $e \in \G$ denote the identity element. Define a probability measure on $\Omega_e^{[0,\infty)}$ as follows. First let $\bfN$ be a Poisson random variable with mean $u|S|/(T+1)$. For each $1\le i \le \bfN$, let $\bfw_i \in \cW_e[0,\infty)$ be a random walk with law $P^{(T)}_{e}$. Now let $\bfom = \sum_{i=1}^\bfN \delta_{\bfw_i}$.  Let $\eta$ be the law of $\bfom$. So $\eta$ is a probability measure on $\Omega_e^{[0,\infty)}$. 

For $\omega \in \Omega_e^{[0,\infty)}$ and $g\in \G$, define 
$$g \ast \omega \in \Omega_g^{[0,\infty)} \textrm{ by } g\ast \omega(\{w\}) = \omega( \{g^{-1}w\}).$$

We claim that $\G \cc (\Omega^{[0,\infty)}, \P_{u,T})$ is measurably conjugate to the Bernoulli shift $\G \cc (\Omega_e^{[0,\infty)}, \eta)^\G$. To see this, denote an arbitrary element $\xi \in (\Omega_e^{[0,\infty)})^\G$ by $\xi=(\xi_g)_{g\in \G}$ and define 
$$\Phi:(\Omega_e^{[0,\infty)})^\G \to \Omega^{[0,\infty)} \textrm{ by } \Phi(\xi) = \sum_{g\in \G} g \ast \xi_g.$$
To check that $\Phi$ is $\G$-equivariant, let $h \in \G$ and $w \in \cW_g[0,\infty)$. Then
\begin{eqnarray*}
\Phi(h\xi)(\{w\}) &=& g \ast (h\xi)_g(\{w\}) = (h\xi)_g(\{g^{-1}w\}) = \xi_{h^{-1}g}(\{g^{-1}w\}) = \xi_{h^{-1}g}(\{g^{-1}hh^{-1}w\}) \\
&=& (h^{-1}g)\ast \xi_{h^{-1}g}(\{h^{-1}w\}) = \Phi(\xi)(\{h^{-1}w\}) = h\Phi(\xi)(\{w\}).
\end{eqnarray*}
Moreover $\Phi$ is invertible with inverse defined by
$$\Phi^{-1}(\omega)_g(\{w\})= \omega(\{gw\})\quad \forall g\in \G, w \in \cW_e[0,\infty).$$
By construction of $\eta$, $\Phi_*\eta^\G = \P_{u,T}$. This proves the claim.

Next we turn towards the entropy computation. If $\bfX$ is any random variable taking values in a countable set $A$ then the {\bf Shannon entropy} of $\bfX$ is
$$H(\bfX):=\sum_{a\in A} -\Prob(\bfX=a) \log (\Prob(\bfX =a)).$$
If $\bfY$ is another random variable taking values in a countable set $B$ then $H(\bfX,\bfY)$ is the Shannon entropy of the join $\bfX \vee \bfY$ which takes values in $A\times B$. We will assume standard facts about entropy and relative entropy that can be found in \cite{MR2239987} for example. We let $\E$ denote expected value. 

Let $\bfN, \bfw_1,\ldots, \bfw_\bfN$ be as above. Let $\bfL_i=\textrm{len}(\bfw_i)$ be the length of $\bfw_i$. Then 
\begin{eqnarray*}
H(\eta) &\le& H(\bfN,\bfL_1,\ldots, \bfL_\bfN, \bfw_1,\ldots, \bfw_\bfN)\\
 &=& H(\bfN) + H(\bfL_1, \ldots, \bfL_\bfN|\bfN)  + H(\bfw_1,\ldots, \bfw_{\bfN}| \bfN, \bfL_1,\ldots, \bfL_\bfN).
\end{eqnarray*}
Next we estimate each of the three terms on the right hand side.

Because $\bfN$ is a Poisson random variable with mean $u|S|/(T+1)=O(u)$, $H(\bfN) = O(u)$. Because the $\bfL_i$'s are i.i.d. and independent of $\bfN$, 
$$H(\bfL_1, \ldots, \bfL_\bfN|\bfN)  = \E[\bfN] H(\bfL_1) = \frac{u|S|}{T+1} H(\bfL_1).$$
Since $\bfL_1+1$ is a geometric random variable with mean $T+1$,
$$H(\bfL_1)  = (T+1)H\left(\frac{1}{T+1}, \frac{T}{T+1}\right) \le \log(2)(T+1)$$
(where $H(x,1-x) := -x \log x - (1-x)\log(1-x)$). Thus
$$H(\bfL_1, \ldots, \bfL_\bfN|\bfN)  =O(u).$$
Because the $\bfw_i$'s are jointly independent and $\bfw_i$ depends only on $\bfL_i$,
$$H(\bfw_1,\ldots, \bfw_{\bfN}| \bfN, \bfL_1,\ldots, \bfL_\bfN)  = \E[\bfN] H(\bfw_1|\bfL_1) = \frac{u|S|}{T+1} H(\bfw_1|\bfL_1).$$
If $\bfL_1=k$ then $\bfw_1$ is uniformly distributed on a set of cardinality $|S|^k$. So $H(\bfw_1|\bfL_1) = \E[\bfL_1] \log |S| = T \log |S|$. So
$$H(\bfw_1,\ldots, \bfw_{\bfN}| \bfN, \bfL_1,\ldots, \bfL_\bfN)  = O(u).$$
The previous three estimates imply $H(\eta)=O(u)$ as required.\end{proof}

\section{The local picture}\label{sec:local}

Let $K \Subset L \Subset V$ be finite subsets of $V$. Let $\cW_K \subset \cW$ be the set of all walks $w$ that visit $K$ at least once (i.e.,  $w(i) \in K$ for some $i$). Define $\cW_L$ similarly. The goal of the next result is to describe the distribution of $\bfom \resto (\cW_L \setminus \cW_K)$ where $\bfom \sim \P_{u,T}$. This will be used later to describe the cluster at the origin by a growth process.

\begin{defn}
Define the $K$-hitting times for $w \in \cW[m,n]$ by
$$H_K(w) = \inf \{ t \in \N:~ w(t) \in K\} \in \N \cup \{\infty\}$$
 $$\tH_K(w) = \inf \{ t \ge 1:~ w(t) \in K\} \in \{1,2,\ldots \} \cup \{\infty\}.$$
By convention the infimum of the empty set is $+\infty$. If $K=\{x\}$ is a singleton then write $H_x(\cdot)$ and $\tH_x(\cdot)$ instead of $H_{\{x\}}(\cdot), \tH_{\{x\}}(\cdot)$ for simplicity.
\end{defn}

Let $\cW^{(2)}$ be the set of all pairs of walks $(a,b) \in \cW[0,\infty) \times \cW[0,\infty)$ such that $a(0)=b(0)$. For $x\in L \setminus K$, let $\zeta^{(T)}_x$ be the measure on $\cW[0,\infty)_x \times \cW[0,\infty)_x \subset \cW^{(2)}$ given by
$$\zeta^{(T)}_x(\{(a,b)\}) = 1_{\tH_L(a) = \infty} P^{(T)}_x(\{a\}) \deg_x 1_{H_K(b) = \infty} P^{(T)}_x(\{b\}).$$
Define a measure $Q_{L,K}^{(T)}$ on $\cW^{(2)}$ by
$$Q_{L,K}^{(T)}= \sum_{x\in L\setminus K} \zeta^{(T)}_x.$$
 The {\bf concatenation map} is defined by
$$\textrm{Con}: \cW^{(2)} \to \cW[0,\infty), \quad \textrm{Con}(a,b) = (a(\textrm{len}(a)), a(\textrm{len}(a)-1), \ldots, a(0), b(1),\ldots, b(\textrm{len}(b))).$$
For example, this means that $\textrm{Con}(a,b)(0)=a(\textrm{len}(a))$. 

\begin{prop}\label{prop:KL}
For $0<u,T<\infty$, let $\bfom \sim \P_{u,T}$ and $K \Subset L \Subset V$ be finite sets. Then the restriction $\bfom \resto (\cW_L \setminus \cW_K)$ is a Poisson point process with intensity measure $u\textrm{Con}_*Q_{L,K}^{(T)}$.
\end{prop}

\begin{proof}
By definition, the restriction $\bfom \resto (\cW_L \setminus \cW_K)$ is a Poisson point process with intensity measure equal to the restriction of $u \nu^{(T)}$ to $\cW_L\setminus \cW_K$. So it suffices to prove that if $w \in \cW_L \setminus \cW_K$ then $\nu^{(T)}(\{w\}) = \textrm{Con}_*Q_{L,K}^{(T)}(\{w\})$. Let $t\ge 0$ be the smallest number such that $w(t)\in L$. Define walks $a,b$ by
$$a = (w(t), w(t-1),\ldots, w(0)), \quad b=(w(t),w(t+1),\ldots, w(\textrm{len}(w))).$$
Then $(a,b) \in \cW^{(2)}$ is the unique pair such that $\textrm{Con}(a,b)=w$ and $Q_{L,K}^{(T)}(\{(a,b)\}) >0$. Indeed, if $(a',b') \in \cW^{(2)}_{<\infty}$ is any pair with $\textrm{Con}(a',b')=w$ then there is an $s \in [0,\textrm{len}(w)]\cap \N$ such that
$$a' = (w(s), w(s-1),\ldots, w(0)), \quad b'=(w(s),w(s+1),\ldots, w(\textrm{len}(w))).$$
Suppose $Q_{L,K}^{(T)}(\{(a',b')\}) >0$. Then $w(s) \in L$. If $s\ne t$ then $\tH_L(a') \le s-t < \infty$. But this implies $Q_{L,K}^{(T)}(\{(a',b')\}) =0$, a contradiction.

Let $x=a(0)=b(0)$. Since $Q_{L,K}^{(T)}(\{(a,b)\}) = \zeta_x(\{(a,b)\})$, it now suffices to prove $\nu^{(T)}(\{w\}) = \zeta_x(\{(a,b)\})$. To simplify the computation, let $\l = \textrm{len}(w), \a = \textrm{len}(a), \b=\textrm{len}(b)$. Then
\begin{eqnarray*}
\nu^{(T)}(\{w\}) &=& \frac{\deg_{w(0)}}{T+1} P^{(T)}_{w(0)}(\{w\})  = \frac{\deg_{w(0)}}{(T+1)^2} \left( \frac{T}{T+1} \right)^{\l} \deg_{w(0)}^{-1} \cdots \deg_{w(\l-1)}^{-1} \\
&=&  \frac{1}{(T+1)^2} \left( \frac{T}{T+1} \right)^{\l} \deg_{w(1)}^{-1} \cdots \deg_{w(\l-1)}^{-1} \\
\end{eqnarray*}
On the other hand,
\begin{eqnarray*}
\zeta_x(\{(a,b)\}) &=& P^{(T)}(\{a\}) \deg_x  P^{(T)}(\{b\}) \\
&=& \frac{1}{(T+1)^2} \left( \frac{T}{T+1} \right)^{\a + \b} \deg_{a(0)}^{-1} \cdots \deg_{a(\a-1)}^{-1} \deg_x  \deg_{b(0)}^{-1} \cdots \deg_{b(\b -1)}^{-1}.
 \end{eqnarray*}
The fact that $\nu^{(T)}(\{w\}) = \zeta_x(\{(a,b)\})$ now follows from: $\l=\a+\b$, $x={b(0)}$ and 
$$(w(1), \ldots, w(\l-1)) = (a(\a-1),\ldots, a(0), b(1),\ldots, b(\b-1)).$$ \end{proof}

\begin{remark}\label{rmk1}
An easy special case of the above result helps aid the intuition. Suppose $K=\emptyset$ and $L=\{x\}$ is a single vertex. Then Proposition \ref{prop:KL} implies: if $\bfom \sim \P_{u,T}$ then the expected value of 
$$\bfom(\{w \in \cW[0,\infty):~ H_x(w) <\infty\})$$
is $u \deg_x \P^{(T)}_x(\tH_x = \infty).$ In particular, if $G$ is vertex transitive then this number is a constant multiple of $u$.
\end{remark}

The next corollary provides a more intuitive description of the restriction  $\bfom \resto (\cW_L \setminus \cW_K)$. It follows immediately from Proposition \ref{prop:KL}.

\begin{cor}\label{cor:KL}
Let $\bfom \sim \P_{u,T}$ and $K \Subset L \Subset V$ be finite sets. For every $x\in L\setminus K$, let $\bfN_x$ be a Poisson random variable with mean $u\deg_x$. For $i=1,\ldots, \bfN_x$, choose a pair $(\bfa_{x,i},\bfb_{x,i}) \in \cW_x^{[0,\infty)} \times \cW_x^{[0,\infty)}$ at random with law $P^{(T)}_x \times P^{(T)}_x$. Let $\bfR_x \subset \{1,\ldots, \bfN_x\}$ be the set of all $j$ such that $\tH_L(\bfa_{x,j})=\infty$ and $H_K(\bfb_{x,j}) = \infty$. Then
$$\bfom \resto (\cW_L \setminus \cW_K) \sim  \sum_{x\in L\setminus K} \sum_{j \in \bfR_x} \delta_{\textrm{Con}(\bfa_{x,j},\bfb_{x,j})}$$
where $\d$ denotes the Dirac delta measure.
\end{cor}

\subsection{The cluster at the origin}\label{sec:cluster1}

This section describes the cluster at the origin in terms of a growth process. 



\begin{defn}
If $\omega \in \Omega^{[0,\infty)}$ then let $\cV_\omega=\cV(\omega)$ be the set of vertices traversed by some walk in the support of $\omega$. Similarly, $\cE_\omega=\cE(\omega)$ is the set of edges $e \in E$ traversed by some walk in the support of $\omega$. 
\end{defn}

Fix $x \in V$. Given $\omega \in \Omega^{[0,\infty)}$, let $\omega^x_1 = \omega \resto \cW_x$ be the restriction of $\omega$ to $\cW_x$ (where $\cW_x = \cW_{\{x\}} \subset \cW$ is the set of walks that contain $x$). For $n\ge 2$, inductively define $\omega^x_n$ by $\omega^x_n = \omega^x \resto \cW_{\cV(\omega^x_{n-1})}$. Finally, define $\omega^x_\infty= \sup_n \omega^x_n$. Then $(\cV(\omega^x_\infty), \cE(\omega^x_\infty))$ is the {\bf $\omega$-cluster containing $x$}. Let $\P^x_{u,T}$ denote the law of $\bfom^x_\infty$ where $\bfom \sim \P_{u,T}$. So $\P^x_{u,T}$ is a probability measure on $\Omega^{[0,\infty)}$.

\begin{prop}\label{prop:cluster1}
Let $\bfom \sim \P_{u,T}$ and $n \ge 0$. Then $\bfom^x_{n+1}$ conditioned on $\cV(\bfom_n)$ and $\cV(\bfom_{n-1})$ is a Poisson point process with intensity measure $u\textrm{Con}_*Q_{\cV(\bfom_n),\cV(\bfom_{n-1})}^{(T)}$. 
\end{prop}

\begin{proof}
This follows from Proposition \ref{prop:KL} and the Poisson nature of $\P_{u,T}$.  \end{proof}

\section{Random walks on non-amenable graphs}

To prepare for the proof of Theorem \ref{thm:main1} (2), this section recalls facts about random walks on non-amenable graphs.

\begin{defn}[Simple random walk]
For $x\in V$, let $P_x$ denote the law of the simple random walk started at $x$. Also let $\E_x$ denote expectation with respect to $\P_x$.
\end{defn}

\begin{defn}[Spectral radius]
Let  $\bfw \sim P_x$ and
$$\rho := \limsup_{n\to\infty} P_x(\bfw(n)=x)^{1/n}$$
be the {\bf spectral radius} of the simple random walk. Because $G$ is connected, $\rho$ does not depend on the choice of $x$. 
\end{defn}

\begin{defn}[Non-amenability]
The graph $G$ is {\bf non-amenable} if $\rho<1$. For example, if $G$ is a Cayley graph of a finitely generated group $\G$ then $G$ is non-amenable if and only if $\G$ is non-amenable \cite[Theorem 10.6]{woess2000random}. 
\end{defn}

The next lemma will be used in the sequel to obtain lower bounds on the growth of the level-$n$ random cluster $\cV(\bfom_n^x)$. 

\begin{lem}\cite[Lemma 2.1]{MR2773031}\label{lem:bnp}
For any $K \Subset V$,
$$\sum_{x\in K} P_x[\tH_K = \infty]  \ge (1-\rho) \# K.$$
\end{lem}

\begin{cor}\label{cor:bnp}
For any $K \Subset V$,
$$\sum_{x\in K} P_x[\tH_K = \infty]^2  \ge (1-\rho)^2 \#K.$$
\end{cor}

\begin{proof}
By Cauchy-Schwartz and Lemma \ref{lem:bnp},
$$\left(\sum_{x\in K} P_x[\tH_K = \infty]^2\right) \left(\sum_{x\in K} 1\right)  \ge \left( \sum_{x\in K} P_x[\tH_K = \infty]\right)^2 \ge (1-\rho)^2 (\# K)^2.$$
Dividing both sides by $\#K$ proves the corollary. \end{proof}


\begin{lem}\label{lem:speed}
Suppose $G$ is non-amenable and connected. Then there exist constants $T_0,\d_0,\s>0$ such that for every $x\in V$ if $T \ge T_0$ then 
$$P^{(T)}_x\big(\textrm{len}(\bfw) \ge T, ~d(x,\bfw(T)) \ge \s T \big) \ge \d_0.$$
\end{lem}

\begin{proof}
By \cite[Lemma 8.1 (b)]{woess2000random}, there is a constant $C>0$ such that 
$$P_x(\bfw(n)=y) \le C \rho^n$$
for all $x,y$. Summing over all $y$ in the ball of radius $\s n$ centered at $x$ yields
$$P_x( d(\bfw(n),x) \le \s n)) \le C D^{\s n} \rho^n.$$
Choose $\s>0$ so that $D^\s \rho <1$ and $T_0$ so that $T\ge T_0$ implies $C D^{\s T} \rho^T \le 1/2$. If $T \ge T_0$ then
\begin{eqnarray*}
P^{(T)}_x\big(\textrm{len}(\bfw) \ge T, ~d(x,\bfw(T)) \ge \s T \big) &=& P^{(T)}_x\big( d(x,\bfw(T)) \ge \s T | \textrm{len}(\bfw) \ge T \big) P^{(T)}_x\big(\textrm{len}(\bfw) \ge T\big) \\
&=& P_x\big( d(x,\bfw(T)) \ge \s T \big) P^{(T)}_x\big(\textrm{len}(\bfw) \ge T\big)\\
 &\ge& (1/2) P^{(T)}_x\big(\textrm{len}(\bfw) \ge T\big).
\end{eqnarray*}
Note
$$P^{(T)}_x\big(\textrm{len}(\bfw) \ge T\big) = \frac{1}{T+1} \sum_{n \ge T} \left( \frac{T}{T+1}\right)^n  =  \left( \frac{T}{T+1}\right)^T = \left( 1-   \frac{1}{T+1}\right)^T \to e^{-1}$$
as $T \to \infty$. So by choosing $T_0$ larger if necessary, we may assume $T \ge T_0$ implies $P^{(T)}_x(\textrm{len}(\bfw) \ge T)  \ge 1/3$.  Set $\d_0=1/6$ to finish the lemma.\end{proof}

\section{A truncated growth process}

The purpose of this section is to study a certain growth process related to $\{\cV(\bfom^x_n)\}_n$. It is the main construction in our proof of Theorem \ref{thm:main1} part (2) which is completed in the next section.

For simplicity, we will assume thoughout that $T$ is a natural number. We also make the following assumptions:
\begin{enumerate}
\item $G=(V,E)$ is a connected graph.
\item $\deg_x \le D$ for all $x\in V$ for some constant $D>0$. 
\item $G$ is non-amenable. Let $\rho<1$ be its spectral radius and $\d_0, \s, T_0$ be the constants in Lemma \ref{lem:speed}.
\item Let $K \Subset L \Subset V$ be finite sets.
\item As in Corollary \ref{cor:KL}, for each $x \in L\setminus K$ let $\bfN_x$ be a Poisson random variable with mean $u\deg_x$ and $(\bfa_{x,i}, \bfb_{x,i}) \sim P^{(T)}_x \times P^{(T)}_x$ for $1\le i \le \bfN_x$.
\item Let $\bfR'$ be the set of all $x \in L \setminus K$ such that $\bfN_x \ge 1$, $\tH_L(\bfa_{x,1})=\infty$, $H_K(\bfb_{x,1}) = \infty$, $\textrm{len}(\bfa_{x,1}) \ge T$ and $d_G(\bfa_{x,1}(T),x) \ge \s T$.
\item Let $\bfa'_{x,i}$ be the walk defined by $\textrm{len}(\bfa'_{x,i}) =  \min(T, \textrm{len}(\bfa_{x,i}))$ and
$$\bfa'_{x,i}(k) = \bfa_{x,i}(k) \quad 0 \le k \le \min(T, \textrm{len}(\bfa_{x,i})).$$
\item Let $\bfom' := \sum_{x \in \bfR'}  \delta_{\bfa'_{x,1}}$. 
\end{enumerate}

The goal of this section is to prove:
\begin{prop}\label{prop:chevy}
There exist constants $\d_1,\d_2>0$ depending only on $G$ such that if $T>T_0$, $\#K \le (1-\rho)^2\#L/2$ and $0<u<\d_1$ then  
$$\P_{u,T} [ \#\cV(\bfom' ) \ge \d_2 uT \#L/2 ] \ge 1 - \frac{4D^{3T}}{u^2T^2 \#L}.$$
To emphasize, the constants $\d_1,\d_2$ do not depend on $K,L,u$ or $T$.
\end{prop}


Proposition \ref{prop:chevy} follows from  Chebyshev's inequality once we have obtained appropriate bounds on the mean and variance of $\#\cV(\bfom')$ which is the purpose of the next lemmas.

\begin{lem}\label{lem:hack}
There exists a universal constant $\d'_1>0$ such that if $0<u<\d'_1$ then
$$(u/2)\d_0[(1-\rho)^2\#L -\#K] \le \E\left[ \# \bfR' \right].$$
\end{lem}

\begin{proof}
By definition of $\bfR'$,
$$\E[ \#\bfR' ] = \sum_{x\in L\setminus K} \P(x \in \bfR').$$
If $\bfw \sim\P^{(T)}_x$ then
\begin{eqnarray*}
&&\P(x \in \bfR') \\
&&=\P(\bfN_x \ge 1) P^{(T)}_x\big(\tH_L(\bfw)=\infty, \textrm{len}(\bfw) \ge T, d_G(\bfw(T),x) \ge \s T \big)P^{(T)}_x(H_K(\bfw) = \infty).
\end{eqnarray*}
We assume $\d'_1>0$ is chosen so that if $u<\d'_1$ then $u/2 \le 1- \exp(-u)$. Assuming $u<\d'_1$ implies
$$u/2 \le 1 - \exp(-u\deg_x) = \P(\bfN_x \ge 1).$$
The next inequalities estimate the middle term in $\P(x\in \bfR')$:
\begin{eqnarray*}
&&P^{(T)}_x\big(\tH_L(\bfa_{x,j})=\infty, \textrm{len}(\bfa_{x,j}) \ge T, d_G(\bfa'_{x,j}(T),x) \ge \s T \big)\\
 &\ge& P^{(T)}_x\big(\tH_L=\infty\big)P^{(T)}_x \big(\textrm{len}(\bfa_{x,j}) \ge T, d_G(\bfa'_{x,j}(T),x) \ge \s  T\big)\\
&\ge&  P_x(\tH_L=\infty) \d_0.
\end{eqnarray*}
Finally, we estimate the last term in $\P(x\in \bfR')$
$$P^{(T)}_x(H_K(\bfb_{x,j}) = \infty) = P^{(T)}_x(H_K=\infty) \ge P^{(T)}_x(\tH_L=\infty) \ge P_x(\tH_L=\infty).$$
Combine these estimates to obtain
\begin{eqnarray*}
\E[\#\bfR'] &=& \E\left[\sum_{x\in L\setminus K} \# \bfR'_x\right] \ge (u/2)\d_0\sum_{x\in L \setminus K} P_x(\tH_L=\infty)^2 \\
&\ge& (u/2)\d_0\left(\sum_{x\in L} P_x(\tH_L=\infty)^2  - \#K\right) \ge u\d_0[(1-\rho)^2\#L - \#K]
\end{eqnarray*}
where the last inequality follows from Corollary \ref{cor:bnp}. \end{proof}

\begin{lem}\label{lem:intersect2}
With notation as in Lemma \ref{lem:hack}, let
$$\bfI=\{(x,y,t,s): ~x,y \in \bfR', ~x\ne y, ~t,s \in [0,T]\cap \N, ~\bfa'_{x,1}(t)=\bfa'_{y,1}(s)\}.$$ 
There is a constant $C>0$ such that $\E[\#\bfI] \le C u^2 \#L.$
\end{lem}

\begin{proof}
The Greens function is defined by: $G(x,y)=\E_x[ \#\{n\ge 0:~\bfw(n)=y\}]$. Because $G$ is non-amenable, $G$ defines a bounded linear operator on $\ell^2(V)$ \cite[Theorem 10.3]{woess2000random}. Consider running independent random walks $\bfw^{(i)}_x$ ($i\in \{1,2\}, x \in L$). The expected number of quadruples $(x,y,s,t)$ with $\bfw^{(1)}_x(s)=\bfw^{(2)}_y(t)$ is $\langle G1_L, G1_L\rangle.$
In particular, conditioned on $\bfR'$,
$$\# \bfI \le \langle G 1_{\bfR'}, G 1_{\bfR'} \rangle.$$
The probability that a given vertex $x \in L$ is in $\bfR'$ is bounded by $1-e^{-uD}$. So if we take the expectation in $\bfR'$ of the above inequality, we obtain
$$\E[\# \bfI] \le (1-e^{-uD})^2 \langle G 1_L, G 1_L \rangle \le (1-e^{-uD})^2\|G\|^2 \|1_L\|_2^2  \le u^2 D^2\|G\|^2 \#L.$$\end{proof}

\begin{lem}\label{lem:1moment}
 There are constants $\d_1,\d_2>0$ depending only on $G$ such that if $T>T_0$, $\#K \le (1-\rho)^2\#L/2$ and $0<u<\d_1$ then
$$\E [ \#\cV(\bfom' )] \ge \d_2 uT  \#L.$$
To emphasize, the constants $\d_1,\d_2$ do not depend on either $K,L,u$ or $T$.
\end{lem}

\begin{proof}
For each $r\in \bfR'$,  $d(\bfa'_r(0), \bfa'_r(T)) \ge \s T$. So the total number of vertices in the range of $\bfa'_r$ is at least $\s T$. It follows that
$$\E [ \#\cV(\bfom' )] \ge \s T \E[\# \bfR'] - \E[\# \bfI]$$
where $\bfI$ is as in Lemma \ref{lem:intersect2}. Assume $u \le \d'_1$. So Lemmas \ref{lem:hack} and \ref{lem:intersect2} imply
$$\E [ \#\cV(\bfom' )] \ge \s Tu\d_0[(1-\rho)^2\#L -\#K]/2  - Cu^2  \#L.$$
So if $\#K \le (1-\rho)^2\#L/2$ and $u<C^{-1} \s T_0 \d_0(1-\rho)^2/4$ then
$$\E [ \#\cV(\bfom' )] \ge \s T u \d_0(1-\rho)^2  \#L/4=: \d_2 uT \#L$$
where $\d_2 :=\s \d_0 (1-\rho)^2/4$. Choose $\d_1 = \min( \d'_1, C^{-1} \s T_0 \d_0(1-\rho)^2/4)$ to finish the lemma. \end{proof}

\begin{lem}\label{lem:2moment}
$$\Var [ \#\cV(\bfom' )] \le D^{3T} \#L.$$
\end{lem}

\begin{proof}
Since $\#\cV(\bfom' ) = \sum_{x\in V \setminus K} 1_{x\in \cV(\bfom')}$,
$$ \Var [ \#\cV(\bfom' )] = \sum_{x, y \in V \setminus K} \textrm{Cov}( 1_{x\in \cV(\bfom')}, 1_{y\in \cV(\bfom')})$$
where $\textrm{Cov}(\cdot,\cdot)$ denotes covariance. Every walk in the support of $\bfom'$ has length $T$. So if $d(x,y)>2T$ then $1_{x\in \cV(\bfom')}, 1_{y\in \cV(\bfom')}$ are independent and $\textrm{Cov}( 1_{x\in \cV(\bfom')}, 1_{y\in \cV(\bfom')})=0$. Also $x\in \cV(\bfom')$ implies $x \in B_T(L)$, the radius $T$-neighborhood of $L$. Thus
$$\Var [ \#\cV(\bfom' )] = \sum_{x \in B_T(L) \setminus K} \sum_{d(x,y)\le 2T} \textrm{Cov}( 1_{x\in \cV(\bfom')}, 1_{y\in \cV(\bfom')}).$$
Since $1_{x\in \cV(\bfom')}$ is either $0$ or $1$, 
$$|\textrm{Cov}( 1_{x\in \cV(\bfom')}, 1_{y\in \cV(\bfom')})| \le 1.$$
Therefore,
$$\Var [ \#\cV(\bfom' )]  \le \#B_T(L) D^{2T} \le D^{3T} \#L.$$\end{proof}

Proposition \ref{prop:chevy} follows from Lemmas \ref{lem:1moment}, \ref{lem:2moment} and Chebshev's inequality.


\subsection{At least one infinite cluster}
The next result strengthens Theorem \ref{thm:main1} (2).
\begin{thm}\label{thm:main1.2}
Let $G=(V,E)$ be a connected  bounded degree graph with spectral radius $\rho<1$.  Then for every $u>0$ there exists $T_u>0$ such that if $T>T_u$ then for $\P_{u,T}$-a.e. $\omega$, the subgraph $(V,\cE_\omega)$ has infinite components.
\end{thm}

\begin{proof}

Let $D$ be the maximum degree of a vertex in $G$. Let $T_0,\s,\d_0,\d_1,\d_2>0$ be the constants from Lemmas \ref{lem:speed} and \ref{lem:1moment}.

Recall that $\P_{u,T}$ is the law of the Poisson point process on $\Omega^{[0,\infty)}$ with intensity measure $u \nu^{(T)}$. So if $u_1, u_2 >0$ are any positive numbers and $\bfom_1, \bfom_2$ are independent random variables with laws $\P_{u_1,T}, \P_{u_2,T}$ respectively, then $\bfom_1 + \bfom_2$ has law $\P_{u_1+u_2,T}$. Since $(V,\cE_{\bfom_1})$ is a subgraph of  $(V,\cE_{\bfom_1+\bfom_2})$, it follows that if $(V,\cE_\omega)$ has an infinite cluster for $\P_{u_1,T}$-a.e. $\omega$ then $(V,\cE_\omega)$ also has an infinite cluster for $\P_{u_1+u_2,T}$-a.e. $\omega$.  So it suffices to prove the theorem in the special case in which $0<u<\d_1$. Assume this from now on.

Choose $T$  large enough so that $T>T_0$ and 
$$(1+\d_2 uT/2) \ge \frac{2}{(1-\rho)^2}.$$

Let $M \ge \frac{2}{(1-\rho)^2}$ be large enough so that 
$$\sum_{n=2}^\infty \frac{4D^{3T}}{u^2T^2 (1+\d_2 uT/2)^{n-2}  M} < 1/2.$$

Let $x \in V$ and recall the definition of $\omega^x_n$ from \S \ref{sec:cluster1}. Also set $\omega^x_0$ equal to the zero measure. Define events $E_n \subset \Omega^{[0,\infty)}$ by
\begin{itemize}
\item $E_1$ is the subset of all $\omega \in \Omega^{[0,\infty)}$ such that  $\#\cV(\omega_1^x) \ge M$;
\item for $n\ge 2$, $E_n$ is the subset of all $\omega \in \Omega^{[0,\infty)}$ such  that $\# \cV(\omega_n^x) \ge (1+\d_2 uT/2) \#\cV(\omega_{n-1}^x)$.
\end{itemize}
If $n\ge 2$ and $\omega \in \cap_{i<n} E_i$, then
$$\#\cV(\omega_{n-1}^x) \ge (1+\d_2 uT/2)^{n-2}  M$$
and
$$\#\cV(\omega_{n-2}^x) \le (1-\rho)^2\#\cV(\omega_{n-1}^x) /2.$$
Apply Proposition \ref{prop:chevy} with $K=\cV(\omega_{n-2}^x), L=\cV(\omega_{n-1}^x)$ to obtain 
$$\P_{u,T}(E_n| \cap_{i<n} E_{i}) \ge 1 - \frac{4D^{3T}}{u^2T^2  \#\cV(\omega_{n-1}^x)} \ge 1 - \frac{4D^{3T}}{u^2T^2 (1+\d_2 uT/2)^{n-2}  M}.$$
The reason this works is that if $\bfom \sim \P_{u,T}$ then $V_n(\bfom) \supset \cV(\bfom')$ by Corollary \ref{cor:KL} and Proposition \ref{prop:cluster1}. So
\begin{eqnarray*}
\P_{u,T}(\cap_{n\ge 1} E_n) &=& \P_{u,T}(E_1)\prod_{n=2}^\infty \P_{u,T}(E_n | \cap_{i<n} E_{i}) \\
&\ge&  \P_{u,T}(E_1)  \prod_{n=2}^\infty 1 - \frac{4D^{3T}}{u^2T^2 (1+\d_2 uT/2)^{n-2}  M} \\
&\ge& \P_{u,T}(E_1) \left(1-  \sum_{n=2}^\infty \frac{4D^{3T}}{u^2T^2 (1+\d_2 uT/2)^{n-2}  M}\right)\\
&\ge& \P_{u,T}(E_1)/2 >0.
\end{eqnarray*}
For any  $\omega \in \cap_{n\ge 1} E_n$ the cluster $\cV(\omega_\infty^x)$ containing $x$ is infinite. So this proves the theorem.\end{proof}

\section{Infinitely many infinite clusters}

The purpose of this section is to prove the following strengthening of Theorem \ref{thm:main1} (3):
\begin{thm}\label{thm:main1.3}
Let $G=(V,E)$ be a vertex-transitive connected graph so that every vertex of $G$ has degree at most $D$. Let $u,T>0$ and suppose that $0<1+2uD< \rho^{-1}$ where $\rho$ is the spectral radius. Also suppose that for $\P_{u,T}$-a.e. $\omega$, the subgraph $(V,\cE_\omega)$ has infinite components. Then for $\P_{u,T}$-a.e. $\omega$, $(V,\cE_\omega)$ has infinitely many infinite clusters.
\end{thm}

First we show that the cluster at the identity is stochastically dominated by a specially designed multi-type branching random walk (mBRW). The trace of this mBRW coincides with a branching random walk (BRW) which is shown to be transient via \cite{MR2284404}. The theorem then follows from a percolation-theory argument due to Benjamini-Schramm \cite{MR2883387}.

\subsection{A random colored forest}
Let $u,D>0$. We will construct a random colored forest $\bfcF$ but first we describe its general structure. The vertex set, denoted  $V^\bfcF$, is a disjoint union $V^\bfcF = \cup_{n\ge 1} V_n^\bfcF$ and  every edge is of the form $(v,v')$ with $v \in V_n^\bfcF$ and $v' \in V_{n+1}^\bfcF$ for some $n$.  In this case we call $v'$ a {\bf child} of $v$ and $v$ is the {\bf parent} of $v'$. We call $V_n^\bfcF$ the {\bf $n$-th generation} of the forest. Every vertex is assigned a color, either red, green or blue.

The forest is constructed as follow. The cardinality of the 1st generation is a Poisson random variable with mean $uD$. Every vertex in the 1st generation is green.

Now suppose that the $n$-th generation has been construction. Then every green vertex has exactly two children, one red and one blue. Every red vertex has one red child, no blue children, and a $\textrm{Poisson}(uD)$-random number of green children. Similarly, every blue vertex has one blue child, no red children, and a $\textrm{Poisson}(uD)$-random number of green children. The random variables in the $n$-th stage of the construction are independent of each other and of all previously constructed random variables.

We denote the random colored forest by $\bfcF=(V^\bfcF, E^\bfcF, \chi^\bfcF)$ where $V^\bfcF$ is the set of vertices, $E^\bfcF$ is the set of edges and $\chi^\bfcF:V^\bfcF \to \{red,green,blue\}$ is the coloring function.

\subsection{The random map into $G$}

As usual, $G=(V,E)$ is a connected graph with all vertex degrees bounded by $D$. Fix $x\in V$. Let $\bfphi:V^\bfcF \to V$ be the random map defined by:
\begin{itemize}
\item $\bfphi(v)=x$ for every $v$ in the first generation,
\item If $(v,v') \in E^\bfcF$ and $v'$ is green then $\bfphi(v)=\bfphi(v')$
\item If $(v,v') \in E^\bfcF$ and $v'$ is not green then $\bfphi(v')$ is a uniformly random neighbor of $\bfphi(v)$. Moreover, we require that the random variable $\bfphi(v')$ is conditionally independent of all previously constructed random variables (and all other random variables constructed on this stage) conditioned on $\bfphi(v)$.
\end{itemize}

The pair $(\bfcF,\bfphi)$ is a multi-type Branching Random Walk (mBRW).

\begin{lem}\label{lem:sd1}
Recall from \S \ref{sec:cluster1} that $\P^x_{u,T}$ denotes the law of $\bfom^x_\infty$ where $\bfom \sim \P_{u,T}$. Let $G,x$ and $(\bfcF,\bfphi)$ be the mBRW  constructed above. Then there exists a random variable $\bftom \sim \P^x_{u,T}$ taking values in $\Omega^{[0,\infty)}$ and coupled with $(\bfcF,\bfphi)$ such that 
\begin{itemize}
\item $\bftom\sim \P^x_{u,T}$,
\item $\bfphi(E^\bfcF) \supset \cE(\bftom)$ a.s.
\end{itemize}
\end{lem}

\begin{proof}

Let $\cG \subset V^\bfcF$ be the subset of green vertices. Let $\{\bftau^b_v, \bftau^r_v\}_{v\in \cG}$ be an i.i.d. family of random variables such that $\bftau^b_v +1 \sim \Geom(T+1)$ and $\bftau^r_v+1 \sim \Geom(T+1)$.  

For each $v \in \cG$, the {\bf red branch} starting at $v_0$ is the unique path $R_v:=(v,R^v_1,\ldots, R^v_k)$ in the forest $(V^\bfcF,E^\bfcF)$   satisfying
\begin{itemize}
\item $\chi^\bfcF(R^v_i)$ is red for all $1\le i \le k$
\item $(v,R^v_1)$ and $(R^v_i,R^v_{i+1})$ are edges of $E^\bfcF$ for all $i$
\item $k = \bftau^r_{v}$.
\end{itemize}
Blue branches are defined similarly. Let $\bfphi(R_v)=(\bfphi(v),\bfphi(R^v_1),\ldots, \bfphi(R^v_k)) \in \cW[0,\infty)$ be the associated random walk.

Let $\cG_1 \subset \cG$ be the set of all green vertices $v \in V^\bfcF_1$ such that $\tH_x(\bfphi(R_v))=\infty$.  In other words, $\bfphi(R^v_i) \ne x$ for all $i>0$.

Let 
$$\bftom_1 := \sum_{v\in \cG_1} \d_{\textrm{Con}(\phi(R_v), \phi(B_v))}.$$
Also let $\bfom \sim \P_{u,T}$. By Proposition \ref{prop:cluster1}, the law of $\bftom_1$ is the same as the law of $\bfom_1^x$.

Suppose $\bftom_n$ and $\cG_n \subset \cG \cap V^\bfcF_n$ have been defined so that $\bftom_{k}$ has the same law of $\bfom_k^x$ for all $1 \le k \le n$. Let  $\cG_{n+1}$ be the set of all vertices $v \in \cG \cap V^\bfcF_{n+1}$ such that  
$$\tH_{\cV(\bftom_{n})}(\bfphi(R^v)) = \infty \textrm{ and } H_{\cV(\bftom_{n-1})}(\bfphi(B^v)) = \infty.$$
 In other words, $v \in \cG_{n+1}$ if the image of the red branch does not return to $\cV( \bftom_n)$ and the image of the blue branch does not traverse $\cV(\bftom_{n-1})$. Now let
$$\bftom_{n+1} := \bftom_n + \sum_{v\in \cG_{n+1}} \d_{\textrm{Con}(\phi(R_v), \phi(B_v))}.$$
By induction and Proposition \ref{prop:cluster1}, the law of $\bftom_{n+1}$ is the same as the law of $\bfom^x_{n+1}$. Let $\bftom = \sup_n \bftom_n$. By construction, $\bftom\sim \P^x_{u,T}$ and $\bfphi(E^\bfcF) \supset \cE(\bftom)$ a.s.\end{proof}

\subsection{Branching random walk}

The next step is to show that the image of the mBRW constructed above agrees with a certain branching random walk. Let $\theta$ be a probability distribution on $\N$. The {\bf Galton-Watson Tree with offspring distribution $\theta$} is the random tree $\bfcT = (V^\bfcT,E^\bfcT)$ constructed as follows. As in the case of the random colored forest there is a partition $V^\bfcT = \cup_{n\ge 0}V_n^\bfcT$ of the vertices and every edge is of the form $(v,v')$ with $v \in V_n^\bfcT, v' \in V_{n+1}^\bfcT$ for some $n$. In this case the index starts at $0$ instead of starting at $1$.

\begin{enumerate}
\item The $0$-th generation consists of a single individual.
\item Suppose the $n$-the generation has been constructed. Then every vertex of the $n$-th generation has a random number of children with law $\theta$. Moreover, these random variables are jointly independent and independent of all the random variables used in the construction of the $n$-th generation.
\end{enumerate}

The {\bf Branching Random Walk (BRW) with offspring distribution $\theta$ started from $x \in V$} is a random pair $(\bfcT,\bfpsi)$ where
\begin{enumerate}
\item $\bfcT$ is a Galton-Watson Tree with offspring distribution $\theta$,
\item  $\bfpsi:\bfcT \to G$ is a graph homomorphism,
\item $\bfpsi(v)=x$ where $v$ is the unique vertex in the $0$-th generation,
\item Suppose $\bfpsi(v)$ has been constructed for every vertex $v$ in the $n$-th generation of $\bfcT$. Then for every $v \in V_n^\bfcT$ and $v'$ with $(v,v') \in E^\bfcT$, let $\bfpsi(v')$ be a uniformly random neighbor of $\bfpsi(v)$. Moreover, we require that the random variable $\bfpsi(v')$ is conditionally independent of all previously constructed random variables (and all other random variables constructed on this stage) conditioned on $\bfpsi(v)$.
\end{enumerate}

The BRW is {\bf transient} if almost surely $\bfpsi^{-1}(v)$ is finite for every $v \in V$. This does not depend on the choice of $x$ if $G$ is connected.

Theorem 3.2 of \cite{MR2284404} implies
\begin{lem}\label{lem:brw1}
If $\bfX \sim \theta$, $\E[\bfX] \le \rho^{-1}$ and $\theta(\{0\})=0$ then the BRW is transient.
\end{lem}

\begin{lem}\label{lem:brw2}
Let $G=(V,E)$ be a connected graph in which all vertex degrees are bounded by $D$. For any $u,T>0$ and $x\in V$ there exists a random tuple of variables $(\bfcF, \bfphi, \bfcT,\bfpsi)$ such that
\begin{itemize}
\item $(\bfcF, \bfphi)$ is a mBRW as in Lemma \ref{lem:sd1}.
\item $(\bfcT,\bfpsi)$ is a BRW started at $x$ with offspring distribution $\theta$ where, if $\bfX \sim \theta$ then $\frac{\bfX-1}{2} \sim \textrm{Poisson}(uD)$. 
\item the images $\bfpsi(\bfcT)$ and $\bfphi(\bfcF)$ agree.
\end{itemize}
\end{lem}

\begin{proof}
Let $(\bfcF,\bfphi)$ be a mBRW started at $x$ as in Lemma \ref{lem:sd1}. Let $\cG \subset V^\bfcF$ be the subset of green vertices. Let $\bfcT'$ be the tree obtained from $\bfcF$ by adding one green vertex at the $0$-th generation and connecting it to every vertex in the 1st generation and then contracting every edge of the form $(v,v')$ such that $\chi^\bfcF(v')=green$. We view $\bfcT'$ as an uncolored tree.  The law of $\bfcT'$ is almost the same as the Galton-Watson tree with offspring distribution. The difference is that the 1st generation of $\bfcT'$ has cardinality $\sim 2\textrm{Poisson}(uD)$ instead of $\sim 1 + 2\textrm{Poisson}(uD)$. Otherwise, the laws are the same. Therefore, it is possible to couple $\bfcT'$ with $\bfcT$ so that $\bfcT'$ is a subtree of $\bfcT$. 

 Set $\bfpsi$ equal to $\bfphi$ on $\bfcT'$. This is well-defined because every edge that gets contracted is mapped under $\bfphi$ to a single vertex.  Extend $\bfpsi$ to all of $\bfcT$ so that $(\bfcT,\bfpsi)$  is a BRW with offspring distribution $\theta$.  The lemma now follows by construction.\end{proof}

\begin{proof}[Proof of Theorem \ref{thm:main1.3}]
This argument is essentially the same as the proof of \cite[Theorem 4]{MR2883387}.

Let $\theta$ be as in Lemma \ref{lem:brw2}. Then the mean of $\theta$ is $1+2uD < \rho^{-1}$. By Lemma \ref{lem:brw1}, the BRW started at $x$ with offspring distribution $\theta$ is transient.

Let $\bfom \sim \P_{u,T}$.  To obtain a contradiction suppose there are pairs of vertices $(x_n,y_n)$ and $\eps>0$ such that $d_G(x_n,y_n) \to \infty$ as $n\to\infty$ and yet the probability that $x_n$ and $y_n$ are in the same cluster of $\bfom$ is at least $\eps$ for all $n$. Since the BRW stochastically dominates the clusters, this means that with probability $\ge \eps$ BRW started from $x_n$ hits $y_n$ and with probability $\ge \eps$ BRW started from $y_n$ hits $x_n$. Consequently, with probability $\ge \eps^2$, the BRW started from $x_n$ will reach $x_n$ again after at least $d_G(x_n,y_n)$ time steps. Since $G$ is transitive, we may assume without loss of generality that all the $x_n$'s are the same. But this contradicts transience of the BRW. This contradiction shows that for every $\eps>0$ there is a $k$ such that $d_G(x,y)>k$ implies that with probability $<\eps$, $x,y$ are in the same cluster of $\bfom$.

Let $r>0$ and consider $m$ balls of radius $r$ in $G$. If $r$ is large then the probability that each ball will intersect an infinite cluster of $\bfom$ nontrivially will be close to 1. However, the probability that any infinite cluster of $\bfom$ intersects two balls tends to zero as the distance between the balls tends to infinity. It follows that, almost surely, there are at least $m$ infinite clusters of $\bfom$. Since $m$ is arbitrary, this proves the theorem. \end{proof}

\begin{remark}\label{rmk1.3}
The proof of Theorem \ref{thm:main1.3} uses vertex-transitivity of $G$ only once, to assume that all of the $x_n$'s are the same. This assumption can be removed as follows: let $x_n,y_n$ be as in the proof above. If $(G_\infty, x_\infty)$ is a sub-sequential Benjamini-Schramm limit of the rooted graphs $(G,x_n)$ then with probability $\ge \eps^2$, the BRW started at $x_\infty$ will reach $x_\infty$ again after $d_G(x_n,y_n)$ time steps for an infinite set of $n$'s. In particular, the BRW is recurrent. However, the spectral radius of $G_\infty$ is bounded by the spectral radius of $G$ because $\rho$ is the supremum of $P_x(\bfw(n)=x)^{1/n}$ as $n\to\infty$ and these functions are continuous with respect to Benjamini-Schramm convergence. This contradicts Lemma \ref{lem:brw1}. The rest of the proof remains the same.
\end{remark}

\section{Infinitely many ends and indistinguishability}
The purpose of this section is to prove the following slight strengthening of Theorem \ref{thm:main1} (4):

\begin{thm}\label{thm:main1.4}
Let $u,T>0$ and $\bfom \sim \P_{u,T}$. Suppose $G$ has a transitive unimodular closed automorphism subgroup $\G \le \Aut(G)$. If $(V,\cE_\bfom)$ has infinitely many infinite clusters a.s. then each infinite cluster has infinitely many ends a.s. Moreover $(V,\cE_\bfom)$ has indistinguishable infinite clusters in the sense of \cite{lyons-schramm-indistinguishability}.
\end{thm}

\begin{defn}
Let $G=(V,E)$ be a graph. Given an edge $e \in E$ and a subset $A \subset E$, we let $\Pi_e(A):=A\cup \{e\}$. Given a collection $\cA \subset 2^{E}$ of subsets, let $\Pi_e(\cA):=\{ \Pi_e(A):~A \in \cA\}$. Now suppose $\bfS$ is a random subset of $E$.  We say that $\bfS$ is {\bf insertion-tolerant} if for every edge $e\in E$, $\P(\bfS \in \cA)>0$ implies $\P(\bfS \in \Pi_e(\cA))>0$. 
\end{defn}

\begin{prop}
For any $u,T>0$, if $\bfom \sim \P_{u,T}$ then $\cE_\bfom$  is insertion tolerant.
\end{prop}

\begin{proof}
This follows from the description of $\P_{u,T}$ as a Poisson point process.\end{proof}

\begin{proof}[Proof of Theorem \ref{thm:main1.4}]
By \cite[Theorem 3.3 and Proposition 3.10]{lyons-schramm-indistinguishability}, insertion tolerance together with unimodular-transitivity of $G$ implies the conclusion. \end{proof}

\begin{remark}\label{rmk:1.4}
Theorem \ref{thm:main1.4} also holds for unimodular networks by  \cite[Corollary 6.11 and Theorem 6.15]{aldous-lyons-unimodular}.
\end{remark}

Theorem \ref{thm:main1} follows from Theorems \ref{thm:main1.1}, \ref{thm:main1.2}, \ref{thm:main1.3} and \ref{thm:main1.4}.

\section{Proof of Theorem \ref{thm:main2}}

In this section, the final pieces are assembled to prove Theorem \ref{thm:main2}. We need to review standard facts from the theory of cost and treeability.

\subsection{Cost and treeability}
Let $(X,\mu)$ be a standard probability space. A Borel equivalence relation $\cR \subset X\times X$ is {\bf measure-preserving} if for every Borel automorphism $\phi:X \to X$ such that $(x,\phi x) \in \cR$ for every $x$ satisfies  $\phi_*\mu = \mu$. The relation $\cR$ is {\bf countable} if its equivalence classes are countable. It is {\bf ergodic} if for every measurable set $A \subset X$ that is a union of $\cR$-classes, $\mu(A) \in \{0,1\}$. For $x\in X$, let $[x]_\cR \subset X$ denote its $\cR$-class.

By Feldman-Moore \cite{MR0578656}, $\cR$ is a countable measure-preserving equivalence relation if and only if there is a countable group $\G$ and a measure-preserving action $\G \cc (X,\mu)$ such that $\cR=\{ (x,gx):~x \in X, g\in \G\}$ is the orbit-equivalence relation.

A {\bf graphing of $\cR$} is a measurable set $\cG \subset \cR$ such that
\begin{itemize}
\item $(x,y) \in \cG \Rightarrow (y,x) \in \cG$;
\item for a.e. $x\in X$, the graph $\cG_x$ with vertex set $[x]_\cR$ and edges $\cG \cap ([x]_\cR \times [x]_\cR)$ is connected.
\end{itemize}
A graphing $\cG$ is a {\bf treeing} if for a.e. $x\in X$, the graph $\cG_x$ is a tree. We say $\cR$ is {\bf treeable} if it admits a treeing.

The {\bf cost} of a graphing $\cG$ is
$$\textrm{cost}(\cG) = \frac{1}{2} \int \#\{y \in [x]_\cR:~(x,y) \in \cG\}~d\mu(x).$$
The cost of $\cR$ is the infimum of $\textrm{cost}(\cG)$ over all graphings $\cG$ of $\cR$. 

The theory of cost was initiated by Levitt \cite{MR1366313} and developed into a powerful tool by Gaboriau \cite{gaboriau-cost}. The main result we need is: if $\cG$ is a treeing then $\cG$ realizes the cost of $\cR$ \cite{gaboriau-cost}. That is, $\textrm{cost}(\cG)=\textrm{cost}(\cR)$. 

The next lemma will help us promote a treeing on part of $X$ to a treeing on all of $X$.

\begin{lem}\label{lem:treeable}
Let $(X,\mu)$ be a standard probability space, $\cR \subset X\times X$ a countable measure-preserving ergodic equivalence relation and $Y\subset X$ a subset with positive measure. Let $\cR_Y = \cR \cap Y \times Y$ be the restriction of $\cR$ to $Y$. Suppose $\cS \subset \cR_Y$ is an ergodic, treeable subequivalence relation with cost $>1$ (with respect to the normalized measure $\mu(Y)^{-1}\mu(\cdot \cap Y)$ on $Y$). Then there exists an ergodic treeable subequivalence relation $\cS' \subset \cR$ with cost $>1$.
\end{lem}

\begin{proof}
Let $\cG$ be a treeing of $\cS$. Also let $\phi:X\setminus Y \to X$ be any measurable map such that
\begin{enumerate}
\item for a.e. $x \in X\setminus Y$, $\phi(x) \in Y$
\item $(x,\phi(x)) \in \cR$ for all $x$.
\end{enumerate}
Such a map exists because $\cR$ is ergodic. Define $\cG'$ to be the union of $\cG$ with the edges $(x,\phi x)$ and $(\phi x,x)$ for $x \in X\setminus Y$. Also let $\cS'$ be the smallest subequivalence relation of $\cR$ that contains $\cG'$. Then $\cG'$ is a treeing of $\cS'$. Moreover the cost of $\cS'$ is equal to $\mu(Y)\textrm{cost}(\cS\resto Y) + \mu(X\setminus Y)$. This is because for every $x\in X \setminus Y$, we are adding an edge that adds one to the degree of $x$ and adds one to the degree of $\phi(x)$. Since $ \textrm{cost}(\cS\resto Y)>1$, this implies $\textrm{cost}(\cS')>1$.

In order to see that $\cS'$ is ergodic, let $A \subset X$ be a measurable union of $\cS'$ classes. Since $\cS' \cap Y \times Y = \cS$, $A \cap Y$ is a measurable union of $\cS$-classes. Since $\cS$ is ergodic,  $\mu(A \cap Y)\in \{0, \mu(Y) \}$. For a.e. $x \in X$, the $\cS'$-class of $x$ intersects $Y$. So the fact that $\mu(A \cap Y)\in \{0, \mu(Y) \}$ implies $\mu(A) \in \{0,1\}$. Since $A$ is arbitrary, $\cS'$ is ergodic.\end{proof}

The next lemma will help us promote the weak vND on a factor action to the weak vND on the action.

\begin{lem}\label{lem:factor1}
Suppose $\G \cc (X,\mu)$ is an ergodic probability-measure-preserving action with an essentially free factor $\G \cc (Y,\nu)$. If this factor is weakly vND  then $\G \cc (X,\mu)$ is also weakly vND.
\end{lem}

\begin{proof}
Because the factor is an essentially free action of $\G$, the factor map $\Phi:X \to Y$ is {\bf class-bijective}. This means that for a.e. $x\in X$, the restriction of $\Phi$ to the orbit $\G x$ is a bijection onto the orbit $\G \Phi(x)$. 

Let $\F_2 \cc (Y,\nu)$ be an essentially free, ergodic action of $\F_2$ whose orbits are contained in the orbit of the $\G$-action. For $x\in X$ and $f\in \F_2$, define $fx \in \G x$ to be the unique element such that $\Phi(fx)=f\Phi(x)$ where $\Phi:X \to Y$ is the factor map. This is well-defined for a.e. $x$ because $\Phi$ is class-bijective. Moreover it defines an essentially free action of $\F_2$ whose orbits are contained in the $\G$-orbits. \end{proof}


\subsection{Strong solidity}

Strong solidity will be used to promote the weak vND to vND. 

\begin{defn}
Let $\G \cc (X,\mu)$ be a probability-measure-preserving action and let $\cR=\{(x,gx):~x\in X, g\in \G\}$ be the orbit-equivalence relation. Then the action is said to be {\bf solidly ergodic} if for any measurable subequivalence relation $\cS \subset \cR$ there exists a measurable partition of $X$ into countably many pieces $X = \sqcup_{i=0}^\infty X_i$ such that 
\begin{enumerate}
\item each $X_i$ is union of $\cS$-classes,
\item $\cS \resto X_0$ is hyperfinite,
\item and $\cS \resto X_i$ is strongly ergodic for all $i>0$. 
\end{enumerate}
This concept was introduced in \cite{MR2647134} and given its name in \cite{gaboriau-icm}.
\end{defn}

\begin{prop}\label{prop:se}
Let $\G \cc (X,\mu)$ be an essentially free, solidly ergodic action. If this action is weakly vND then it is vND. 
\end{prop}

\begin{proof}
Without loss of generality, we may assume $\G$ is non-amenable. Let $\cR$ be the orbit-equivalence relation of this action. By assumption there exists an essentially free action $\F_2 \cc (X,\mu)$ such that if $\cS = \{(x, fx):~x\in X, f\in \F_2\}$ then $\cS \subset \cR$. Because the action is solidly ergodic, there exists a measurable partition of $X$ into countably many pieces $X = \sqcup_{i=0}^\infty X_i$ such that $\cS \resto X_0$ is hyperfinite and $\cS \resto X_i$ is strongly ergodic for all $i>0$. Because $\F_2$ is non-amenable, $X_0$ must be have measure zero. Without loss of generality, $\mu(X_1)>0$. By Lemma \ref{lem:treeable} there exists an ergodic treeable subequivalence relation $\cS' \subset \cR$ with cost $>1$. By \cite[Proposition 14]{gaboriau-lyons}, $\cS'$ contains a subrelation that is the orbit-equivalence relation for an ergodic essentially free action of a rank 2 free group. This last result uses Hjorth's Lemma \cite{hjorth-cost-attained}. \end{proof}

\subsection{The final step}

We will need a recent result of Seward:

\begin{thm}[\cite{seward-ornstein}]\label{thm:isom}
Let $\G$ be a countably infinite group and $(K,\kappa), (L,\lambda)$ probability spaces with the same Shannon entropies $H(\kappa)=H(\lambda)$. Then the corresponding Bernoulli shifts $\Ga \cc (K,\kappa)^\G$ and $\G \cc (L,\l)^\G$ are isomorphic.
\end{thm}

\begin{cor}\label{cor:factor100}
Let $\G$ be a countably infinite group and $(K,\kappa), (L,\lambda)$ probability spaces with $H(\kappa)\ge H(\lambda)$. Then $\Ga \cc (K,\kappa)^\G$ factors onto $\G \cc (L,\l)^\G$.
\end{cor}

\begin{proof}
Let $(N,\nu)$ be another probability space with 
$$H(\kappa) = H(\lambda) + H(\nu) = H(\lambda \times \nu).$$
By Theorem \ref{thm:isom}, $\G \cc (K,\k)^\G$ is measurably conjugate to $\G \cc (L \times N, \l \times \nu)^\G$. The projection map $(L\times N)^\G \to L^\G$ gives the desired factor.\end{proof}

\begin{proof}[Proof of Theorem \ref{thm:main2}]
Without loss of generality, we may assume $\G$ is non-amenable. Because amenability is closed under direct unions, every non-amenable group $\G$ contains a finitely generated non-amenable subgroup $\G'$. Any Bernoulli shift action $\G \cc (K,\k)^\G$ restricts to a Bernoulli action of $\G'$. So we may assume without loss of generality that $\G$ is finitely generated.

By \cite{MR2647134} every Bernoulli shift action of $\G$ is solidly ergodic. By Proposition \ref{prop:se} and Lemma \ref{lem:factor1} it suffices to show that every Bernoulli action $\G \cc (K,\k)^\G$ admits an essentially free factor $\G \cc (X,\mu)$ that is weakly vND. Because Bernoulli shifts are mixing, every nontrivial factor of a Bernoulli shift is essentially free. By Corollary \ref{cor:factor100}, it suffices to prove that for every $\eps>0$ there exists a Bernoulli shift $\G \cc (K,\k)^\G$ with $H(\k)<\eps$ such that this Bernoulli shift admits a nontrivial factor that is weakly vND.

Let $G=(V,E)$ be a Cayley graph of $\G$. By Theorem \ref{thm:main1} there exist  $u,T>0$ such that $\G \cc (\Omega,\P_{u,T})$ is measurably conjugate to a Bernoulli shift with base entropy $<\eps$ and for $\P_{u,T}$-a.e. $\omega$, $(V,\cE_\omega)$ has infinitely many infinite clusters with infinitely many ends. Moreover, $(V,\cE_\bfom)$ has indistiguishable infinite clusters (where $\bfom \sim \P_{u,T}$).

The rest of the proof is essentially the same as in \cite{gaboriau-lyons}. Let $\G$ act on $E$ by left-translation. This induces an action on $2^E$, the space of all subsets of $E$ with the pointwise convergence topology. The map $\omega \mapsto \cE_\omega$ from $\Omega^{[0,\infty)}$ to $2^E$ is $\G$-equivariant. So $\cE_*\P_{u,T}$ is an invariant measure on $2^E$. Let $\cR=\{ (\omega, g\omega):~\omega \in 2^E, g\in \G\}$ be the orbit-equivalence relation.

Let $X_\infty \subset 2^E$ be the set of $\omega \subset E$ such that the identity element is contained in an infinite cluster of $(V,\cE_\omega)$. The {\bf cluster relation}, denoted $\cR^{cl} \subset \cR$, is the set of all pairs $(\omega,g\omega)$ such that the component of the graph $(V,\cE_\omega)$ containing the identity element also contains $g^{-1}$. Because $(V,\cE_\bfom)$ has indistinguishable infinite clusters, $\cR^{cl} \resto X_\infty$ is ergodic with respect to the measure $\cE_*\P_{u,T}$.  

By \cite[Cor. IV.24 (2)]{gaboriau-cost}, because the cluster containing the origin has infinitely many ends with positive probability, the cost of $\cR^{cl}\resto X_\infty$ is larger than 1. By either \cite[Lem.28.11; 28.12]{MR2095154} or \cite[Corollary 40]{pichot-thesis}, $\cR^{cl}\resto X_\infty$  contains an ergodic treeable subequivalence relation $\cF$ with cost larger than 1 (this is obtained from an arbitrary graphing by carefully removing cycles). By Lemma \ref{lem:treeable}, there is an ergodic treeable equivalence relation $\cS \subset \cR$ with cost $>1$. By \cite[Proposition 14]{gaboriau-lyons}, $\cS$ contains a subrelation that is the orbit-equivalence relation for an ergodic essentially free action of a rank 2 free group. This last result uses Hjorth's Lemma \cite{hjorth-cost-attained}.

So the action $\F_2 \cc (2^E, \cE_*\P_{u,T})$ is vND. Since this is a factor of the Bernoulli action $\G \cc (\Omega,\P_{u,T})$ which has base entropy $<\eps$, this completes the proof. \end{proof}

\section{Applications}

\subsection{Actions with positive Rokhlin entropy}\label{sec:app1}

This section proves Theorem \ref{thm:main3}. 

\begin{defn}
Fix a standard probability space $(X,\mu)$. The {\bf Shannon entropy} of a countable measurable partition $\cP$  of $X$ is defined by
$$H_\mu(\cP) := - \sum_{P \in \cP} \mu(P)\log \mu(P).$$
A partition $\cP$ is {\bf generating} for an action $\G \cc (X,\mu)$ if the smallest $\G$-invariant sigma-algebra containing $\cP$ contains all Borel subsets (up to measure zero). The {\bf Rokhlin entropy} of an ergodic probability-measure-preserving action $\G \cc (X,\mu)$ is defined to be the infimum of $H_\mu(\cP)$ over all generating partitions for the action. 
\end{defn}


Seward's generalization of Sinai's Factor Theorem \cite{seward-sinai-30} is:

\begin{thm}[Seward's Factor Theorem]\label{thm:seward-sinai-30}
Let $\G$ be a countably infinite group and let $\G \cc (X,\mu)$ be an essentially free, probability-measure-preserving, ergodic action with positive Rokhlin entropy. Then  $\G \cc (X,\mu)$ factors onto a Bernoulli shift.
\end{thm}

Corollary \ref{thm:main3} follows immediately from Theorem \ref{thm:main2}, Theorem \ref{thm:seward-sinai-30} and Lemma \ref{lem:factor1}.

\subsection{Factors of Bernoulli shifts}\label{sec:app2} 


In this section, we provide details to explain Corollaries \ref{thm:main4} and \ref{thm:main5}. We will need the following lemma proven in \cite[Lemma 6]{bowen-zero-entropy}.
\begin{lem}\label{lem:GL}
Let $\Ga \cc (X,\mu)$ be an essentially free factor of a Bernoulli shift and suppose that its orbit-equivalence relation contains a non-hyperfinite treeable subequivalence relation $\cS$. Then for every pair of probability spaces $(K,\kappa), (L,\lambda)$ the direct product action 
$$\Ga \cc (X \times K^\Ga, \mu \times \kappa^\Ga)$$
factors onto the Bernoulli shift $\Ga \cc (L,\lambda)^\Ga$. 
\end{lem}

\begin{proof}[Proof of Corollary \ref{thm:main4}]
Let $(K,\kappa), (L,\lambda)$ be nontrivial standard probability spaces. For $i=1,2$, let $(K_i,\kappa_i)$ be nontrivial standard probability spaces  such that 
$$H(\kappa)= H(\kappa_1) + H(\kappa_2).$$
By Theorem \ref{thm:isom}, $\G \cc (K,\kappa)^\G$ is measurably conjugate to $\G \cc (K_1\times K_2,\kappa_1 \times \kappa_2)^\G$ which in turn is isomorphic to the direct product action $\G \cc  (K_1,\kappa_1)^\G \times (K_2,\kappa_2)^\G$. By Theorem \ref{thm:main2} and Lemma \ref{lem:GL}, this implies that $\G \cc (K,\kappa)^\G$ factors onto $\G \cc (L,\l)^\G$.\end{proof}

\begin{cor}
Let $\G$ be a non-amenable sofic group. Then the set of all nontrivial factors of Bernoulli shifts over $\G$ forms a single weak isomorphism class. 
\end{cor}

\begin{proof}
It is immediate from the definition of sofic entropy that Rokhlin entropy upper bounds sofic entropy \cite{bowen-survey}. By \cite{kerr-cpe} any nontrivial factor $\G \cc (X,\mu)$ of a Bernoulli shift has positive sofic entropy and therefore positive Rokhlin entropy.  Seward's Factor Theorem \ref{thm:seward-sinai-30} implies that $\G \cc (X,\mu)$ factors onto a Bernoulli shift. By Theorem \ref{thm:main4}, $\G \cc (X,\mu)$ factors onto all Bernoulli shifts and therefore $\G \cc (X,\mu)$ factors onto all factors of all Bernoulli shifts.\end{proof}

\begin{remark}
Popa and Sasyk showed that if $\G$ has property (T) then there exist factors of Bernoulli shifts over $\G$ that are not isomorphic to Bernoulli shifts \cite{popa-sasyk}. 
\end{remark}

Corollary \ref{thm:main5} is an immediate corollary of Theorem \ref{thm:seward-sinai-30} and Corollary \ref{thm:main4}.

\appendix

\section{Convergence to RI}

This section shows how FRI converges to the random interlacement process (RI) in distribution as $T\to\infty$. We will use notation as in \cite[\S 3.2]{MR3773383} to define the RI. The proof of Theorem \ref{thm:RI} below is similar in spirit to \cite[Proposition 3.3]{MR3773383}.

For $K \subset V$ and $w \in \cW[m,n]$, let
$$H^+_K(w) = \sup \{m\le i \le n :~ w(i) \in K\}.$$
For $w \in \cW_K$, let $w_K$ be the restriction of $w$ to the interval $[H_K(w),H_K^+(w)]$. We equip $\cW$ with the topology generated by open sets of the form
$$\{w \in \cW_K:~w_K = w'_K\}$$
where $K \subset V$ is finite and $w' \in \cW_K$. This is a complete Polish topology although it is non-compact if $V$ is infinite since there are infinitely many restrictions $w'_K$. 

The {\bf time shift} $\theta_k:\cW \to \cW$ is defined by $\theta_k:\cW[m,n] \to \cW[m-k,n-k]$,
$$\theta_k(w)(i) = w(i+k).$$
The space $\cW^*$ is defined to be the quotient $\cW^*=\cW/\sim$ where $w_1\sim w_2$ if and only if $w_1=\theta_k(w_2)$ for some $k$. Let $\pi:\cW \to \cW^*$ denote the quotient map and equip $\cW$ with the quotient topology. 

For $w \in \cW$, let $w^\gets \in \cW$ be the time reversed walk defined by $w^\gets(i)=w(-i)$. For $\sA \subset \cW$, let $\sA^\gets=\{w^\gets:~w \in \sA\}$. If $w \in \cW[m,n]$ and $m \le m' \le n' \le n$ then let $w\resto [m',n'] \in \cW[m',n']$ be the restriction. 

For a finite set $K \subset V$, let $Q_K$ be the measure on $\cW$ defined by
$$Q_K(\{w \in \cW:~w(0) \notin K\}) = 0$$
and for each $x \in K$ and Borel subsets $\sA,\sB \subset \cW$,
\begin{eqnarray*}
&&Q_K(\{w \in \cW:~ w\resto(-\infty,0] \in \sA, w(0)=x \textrm{ and } w\resto [0,\infty) \in \sB\})\\
&=& \deg_x P_x( \sA^\gets \cap \tH_K=\infty) P_x(\sB).
\end{eqnarray*}

\begin{thm}[Sznitman \cite{MR2680403} and Teixeira \cite{MR2525105}]\label{thm:ST}
Let $G$ be a transient graph. There exists a unique $\s$-finite measure $Q^*$ on $\cW^*$ such that for every Borel set $\sA \subset \cW^*$ and finite $K \subset V$,
$$Q^*(\sA \cap \pi(\cW_K)) = Q_K(\pi^{-1}(\sA)).$$
\end{thm}

Let $\Omega^*$ be the set of locally finite counting measures on $\cW^*$ with the weak* topology. This means that measures $\omega_i \in \Omega^*$ converge to $\omega_\infty \in \Omega^*$ iff for every finite $K \subset V$ and  $w' \in \cW_K$,
$$\omega_i( \pi (\{w \in \cW_K:~w_K = w'_K\})) \to \omega_\infty( \pi (\{w \in \cW_K:~w_K = w'_K\})).$$

A {\bf random interlacement} with intensity $u$ is a Poisson point process with intensity measure $u Q^*$. Let  $\P^*_u$ be its law. 

\begin{thm}\label{thm:RI}
For any $u>0$, the measures $\P_{u,T}\circ \pi^{-1}$ converge to $\P^*_u$ in the weak* topology as $T \to\infty$. 
\end{thm}

\begin{proof}
By Proposition \ref{prop:KL} and Theorem \ref{thm:ST}, it suffices to show for every finite $K \subset V$, $\textrm{Con}_*Q_{K,\emptyset}^{(T)}$ converges to $Q_K$ as $T \to\infty$. This follows from the fact that $\P_x^{(T)}$ converges to $\P_x$ as $T \to \infty$ (for any $x \in V$). \end{proof}

\bibliography{biblio}
\bibliographystyle{alpha}

\end{document}